\documentclass[11pt]{amsart}
\usepackage{amssymb,amsmath,txfonts}
\usepackage{hyperref}
\newtheorem{theorem}{Theorem}
\newtheorem{prop}{Proposition}
\newtheorem{lemma}{Lemma}
\newtheorem{remark}{Remark}

\newtheorem*{definition}{Definition}
\newtheorem{cor}{Corollary}

\numberwithin{equation}{section}

\def\XXint#1#2#3{{\setbox0=\hbox{$#1{#2#3}{\int}$}
  \vcenter{\hbox{$#2#3$}}\kern-.5\wd0}}

\author{Gang Liu}
\address{Department of Mathematics\\University of Minnesota\\Minneapolis, MN 55455}
\email{liuxx895@math.umn.edu}
\title[Bakry-Emery tensor]{\bf Stable weighted minimal surfaces in manifolds with nonnegative Bakry-Emery Ricci tensor}
\date{}
\begin{document}
\begin{abstract}In this paper, we study stable weighted minimal hypersurfaces in manifolds with nonnegative Bakry-Emery Ricci curvature.
We will give some geometric and topological applications. In particular, we give some partial classification of complete 3-manifolds with nonnegative Bakry-Emery Ricci curvature assuming
that $f$ is bounded.
\end{abstract}\maketitle

\section{\bf{Introduction}}

A smooth metric measure space is a triple $(M, g, e^{-f}dvol)$, where $M$ is a smooth manifold; $g$ is the
Riemannian metric
on $M$; $f$ is a smooth function and $dvol$ is the volume form induced by $g$. This object has been studied
 extensively
in geometric analysis in recent years, e.g, \cite{[P1]}\cite{[Lo]}\cite{[WW]}\cite{[MW]}\cite{[MW1]}\cite{[MW2]}.
Perelman \cite{[P1]} introduces
a functional which involves an integral of the scalar curvature with respect to a weighted measure. The Ricci flow is thus
a gradient
flow of such a functional. Metric measure spaces also arise in smooth collapsed Gromov-Hausdorff limits. In the
physics
literature, $f$ is refered to as the dilation field. On the smooth metric measure space, there is an important
curvature quantity
 called the Bakry-Emery Ricci curvature, which is defined in \cite{[BE]} by
$$Ric_f = Ric + \nabla^2 f.$$
One observes that $Ric_f = \lambda g$ for some constant $\lambda$ is exactly the gradient Ricci soliton equation,
which
plays an essential role in the analysis of the singularities of the Ricci flow.

A lower bound for Bakry-Emery curvature is a natural assumption to make and it has significant geometric
consequences. More generally, $Ric_f$ has a natural extension to metric measure spaces, see \cite{[LV]}\cite{[S1]}\cite{[S2]}. 
Recently, in \cite{[WW]}, G. F. Wei and W. Wylie proved the weighted volume comparison theorems; O. Munteanu and J. Wang
established
the gradient estimate for positive weighted harmonic functions. It should be noted that a while back, Lichnerowicz
 \cite{[Lc]}
has generalized the classical Cheeger-Gromoll splitting theorem \cite{[CG]} to the metric measure spaces with $Ric_f \geq 0$ and
 $f$ is bounded(See \cite{[FLZ]} for more generalizations).

 In Riemannian geometry, minimal surfaces arise naturally in the variation of the area functional.
A minimal surface is called stable if the second variation of the area is nonnegative for any compactly
supported variations.
Minimal surfaces have their own beauties, e.g, Bernstein's theorem.
Moreover, they have important applications to the geometry and topology of manifolds.
For example, more than 60 years ago, the Synge theorem and the Bonnet-Meyers theorem were proved by the variation of geodesics(one
dimensional minimal surface).
More recently, by using minimal surfaces, Schoen and Yau proved the famous positive mass conjecture \cite{[SY2]}\cite{[SY3]}
. Meeks and Yau \cite{[MY1]}\cite{[MY2]} proved the loop theorem, sphere theorem
and Dehn lemma together with the equivariant forms. In \cite{[SY1]}, Schoen and Yau proved that a complete
noncompact 3-manifold with positive Ricci curvature is diffeomorphic to $\mathbb{R}^3$.
Anderson \cite{[An1]} studied the restriction of the first betti number for manifolds with nonnegative Ricci curvature;
the author \cite{[L]} used the minimal surface theory to classify complete three dimensional
manifolds with nonnegative Ricci curvature.

In the study of smooth metric measure spaces, it is natural to add a weight $e^{-f}$ on the area functional
of the surface.
The critical points of the weighted area functional are called weighted minimal surfaces. A weighted minimal
surface is called stable if the second
variation of the weighted area is nonnegative.

Very recently, X. Cheng, T. Mejia and D. T. Zhou \cite{[CZ]} studied the stability condition and compactness of
$f$-minimal surfaces. They \cite{[CZ1]} also gave eigenvalue estimates for certain closed $f$-minimal surfaces.

In this paper, we will investigate some geometric and topological results for smooth metric measure spaces via
analyzing stable weighted minimal surfaces. We shall assume that the Bakry-Emery Ricci curvature is nonnegative.

Below is the organization of this paper.
In section 2, we will derive the second variation formula for the weighted area(see also \cite{[CZ]} and \cite{[Bay]} for the derivation).
We give an application to compact stable $f$-minimal surfaces in section 3. This generalizes some previous works of Heintze and Karcher \cite{[HK]}.
An example is given in section 4 to show that a result of Schoen and Fischer-Colbrie \cite{[FS]} cannot
be extended to the case when Bakry-Emery Ricci curvature is nonnegative.
In section 5 we give an application of the stability inequality to noncompact case.
In section 6, we study the topology of complete 3-manifolds with nonnegative Bakry-Emery Ricci curvature.

\section*{Acknowledgements}

The author would like to express his deep gratitude to his advisor, Professor Jiaping Wang, for his interest in this problem and useful suggestions.
He also thanks Professor Frank Morgan for pointing out more references.

\section{\bf{Second variation formula}}
\begin{definition}
Let $(M^m, g, e^{-f}dv)$ be a complete smooth metric measure space and $\Sigma$ be a complete submanifold in $M$. We say $\Sigma$ is  $f$-minimal in $M$, if the first variation of the $e^{-f}$ weighted area functional vanishes at $\Sigma$. $\Sigma$ is called stable $f$-minimal if the second variation of the $e^{-f}$ weighted area functional is nonnegative along any compactly supported variational normal vector field.
\end{definition}

\begin{prop}
Let $(M^m, g, e^{-f}dv)$ be a complete smooth metric measure space and $\Sigma^n$ be a complete $f$-minimal submanifold in $M$.
 Let $e_i(0\leq i\leq n)$ be an orthonormal frame in an open set of $\Sigma$. Define $\nabla^T$ and $\nabla^\perp$ to be the connections
projected to the tangential and normal spaces on $\Sigma$.
Then $$H = \nabla^{\perp}f$$ where $H = -\sum\limits_{i}\nabla_{e_i}^\perp e_i$ is the mean curvature vector. If $\Sigma_t$($-\epsilon < t < \epsilon$) is a smooth family of the submanifolds such that $\Sigma_0 = \Sigma$ and the variational normal vector field $\nu$ is compactly supported on $\Sigma_t$, then at $t=0$,
$$\frac{d^2\int_{\Sigma_t}e^{-f}}{dt^2} = \int_\Sigma e^{-f}(-\sum\limits_{i=1}^nR_{i\nu\nu i}-\frac{1}{2}\Delta_\Sigma(|\nu|^2)+|\nabla_\Sigma \nu|^2-2|A^\nu|^2
-f_{\nu\nu}+\frac{1}{2}\langle\nabla^Tf, \nabla^T(|\nu|^2)\rangle)$$ where $A^\nu_{ij} = -\langle \nabla_{e_i}e_j, \nu\rangle$.
\end{prop}
\begin{proof}
For any point $p \in \Sigma_0$, consider a local frame $e_i(1\leq i \leq n)$ near $p$ such that they are tangential to $\Sigma_t$ and $[e_i, \nu] = 0$ for all small $t$. We can also assume that at $p$, $e_i$ is an orthonormal frame and $\nabla^{T}_{e_i}e_j = 0$. Let $g_{ij} = \langle e_i, e_j\rangle$ and $g^{ij}$ be the inverse matrix of $g_{ij}$. We have
$$\frac{d\int_{\Sigma_t}e^{-f}}{dt} = \int_{\Sigma_t}e^{-f}\langle H - \nabla^{\perp}f, \nu\rangle$$
where $$H = -(\nabla_{e_i}e_j)^{\perp}g^{ij}.$$ Thus if $\Sigma_0$ is $e^{-f}$ minimal, $$H = \nabla^{\perp}f.$$

At $p$, we have
\begin{equation}
\begin{aligned}
\frac{d\langle H, \nu\rangle}{dt}& = -(\langle\nabla_{\nu}\nabla_{e_i}e_j, \nu\rangle g^{ij} + \langle\nabla_{e_i}e_j, \nabla_\nu\nu\rangle g^{ij} + \langle\nabla_{e_i}e_j, \nu\rangle\nu(g^{ij}))\\&= -(\sum\limits_{i=1}^nR_{\nu ii\nu}
+ \langle\nabla_{e_i}\nabla_\nu e_i, \nu\rangle -\langle H, \nabla_\nu\nu\rangle - \sum\limits_{i, j = 1}^n
\langle\nabla_{e_i}e_j, \nu\rangle(\langle \nabla_{\nu}e_i, e_j\rangle + \langle \nabla_{\nu}e_j, e_i\rangle))\\&=
-(\sum\limits_{i=1}^nR_{\nu ii\nu}+\frac{1}{2}\Delta_\Sigma(|\nu|^2)-\sum_{i=1}^{n}|\nabla_{e_i}\nu|^2+2\sum\limits_{i, j = 1}^n|\langle\nabla_{e_i}e_j, \nu\rangle|^2 -\langle H, \nabla_\nu\nu\rangle).
\end{aligned}
\end{equation}
\begin{equation}
\begin{aligned}
\frac{d\langle \nabla^\perp f, \nu\rangle}{dt}& =\nu\nu (f)\\&=f_{\nu\nu} + \langle\nabla^T f, \nabla_\nu\nu\rangle + \langle\nabla^\perp f, \nabla_\nu\nu\rangle\\&=f_{\nu\nu}+\sum\limits_{i=1}^ne_i(f)\langle e_i, \nabla_\nu\nu\rangle +
\langle\nabla^\perp f, \nabla_\nu\nu\rangle\\&=f_{\nu\nu} - \frac{1}{2}\langle\nabla^T f, \nabla^T(|\nu|^2)\rangle+
\langle\nabla^\perp f, \nabla_\nu\nu\rangle.
\end{aligned}
\end{equation}

Since $\Sigma_0$ is $f$ minimal, by the two equalities above, we have
\begin{equation}
\begin{aligned}
\frac{d^2\int_{\Sigma_t}e^{-f}}{dt^2}& = \frac{d\int_{\Sigma_t}e^{-f}\langle H - \nabla^{\perp}f, \nu\rangle}{dt}
\\&=\int_\Sigma e^{-f}(-\sum\limits_{i=1}^nR_{i\nu\nu i}-\frac{1}{2}\Delta_\Sigma(|\nu|^2)+|\nabla_\Sigma\nu|^2-2|A^\nu|^2
-f_{\nu\nu}+\frac{1}{2}\langle\nabla^Tf, \nabla^T(|\nu|^2)\rangle).
\end{aligned}
\end{equation}

\end{proof}

\begin{cor}
Let $(M^m, g, e^{-f}dv)$ be a complete oriented Riemannian manifold and $\Sigma_t$ be a smooth family of oriented
hypersurfaces in $M$. Let $N$ be the unit normal vector field on $\Sigma_t$. Suppose the variational vector field
for $\Sigma_t$ is given by $\lambda N$ where $\lambda$ is smooth function with compact support on $\Sigma_t$.
If $\Sigma_0$ is $e^{-f}$ minimal, then the mean curvature of $\Sigma_0$ satisfies $$H = f_n.$$ where $f_n$ is  the normal derivative of $f$.
Moreover,
 $$\frac{d^2\int_{\Sigma_t}e^{-f}}{dt^2}|_{t=0} =
\int_{\Sigma_0}(|\nabla \lambda|^2 - \lambda^2(Ric_f(n,n) + |A|^2))e^{-f}$$ where $Ric_f = Ric + \nabla^2 f$, $A$ is the second fundamental form.
Therefore, the stability inequality is $$\int_{\Sigma_0}(|\nabla \lambda|^2 - \lambda^2(Ric_f(n,n) + |A|^2))e^{-f}\geq 0$$
for any compactly supported function $\lambda$ on $\Sigma_0$.
\end{cor}
\begin{proof}
Since $\Sigma_0$ is weighted minimal,  according to Proposition 1, $$H = \langle\nabla^\perp f, N\rangle = f_n.$$
Let $\nu = \lambda N$. For an orthonormal frame $e_i$ at a point on $\Sigma_0$,
\begin{equation}
\begin{aligned}
|\nabla_\Sigma\nu|^2 &= |\langle\nabla_{e_i}(\lambda N), \nabla_{e_i}(\lambda N)\rangle|^2
\\&=|\nabla\lambda|^2+\sum\limits_{i, j}|\langle\nabla_{e_i}(\lambda N), e_j\rangle|^2\\&=
|\nabla\lambda|^2 + \lambda^2|A|^2.
\end{aligned}
\end{equation}
Therefore
\begin{equation}
\begin{aligned}
\frac{d^2\int_{\Sigma_t}e^{-f}}{dt^2}& =\int_{\Sigma_0} e^{-f}(-\sum\limits_{i=1}^nR_{i\nu\nu i}-\frac{1}{2}\Delta_\Sigma(|\nu|^2)+|\nabla_\Sigma\nu|^2-2|A^\nu|^2
-f_{\nu\nu}+\frac{1}{2}\langle\nabla^Tf, \nabla^T(|\nu|^2)\rangle)\\&=\int_{\Sigma_0} e^{-f}
(-\lambda^2Ric_f(n, n)-\lambda\Delta\lambda-\lambda^2|A|^2+\langle\nabla f, \nabla\lambda\rangle\lambda)
\\&=\int_{\Sigma_0}(|\nabla \lambda|^2 - \lambda^2(Ric_f(n,n) + |A|^2))e^{-f}.
\end{aligned}
\end{equation}
In the last step, we have used the integration by parts.
\end{proof}

\section{\bf{An application to the compact case}}
In \cite{[S]}, Simons observed that there are no closed, stable minimal 2-sided hypersurfaces in a manifold with positive Ricci
curvature. Later Heintze and Karcher \cite{[HK]} proved that the exponential map of the normal bundle of a hypersurface
$\Sigma\in M$ is area decreasing, if $\Sigma$ is stable, minimal and $M$ has nonnegative Ricci curvature.
Anderson extended this result, he also proved that a version of the Cheeger-Gromoll splitting theorem in the compact
case, see \cite{[An2]}.  More recently, F. Morgan \cite{[M]} obtained the upper bound of weighted volume of one side of a hypersurface which generalizes some works in \cite{[HK]}. See also chapeter $18$ in \cite{[M1]} for more discussion.

In this section, we shall prove the following:

\begin{theorem}
{Let $(M^m, g, e^{-f}dv)$ be an oriented complete Riemannian manifold and $\Sigma$ be
a closed oriented stable $f$-minimal hypersurface in $M$.
If $Ric_f \geq 0$, then $\Sigma$ is totally geodesic and $Ric_f(n, n) = 0$. If $\Sigma$ is
weighted $f$-area-minimizing in its homology class, then $M^m$ is isometric to a quotient of
$\Sigma \times \mathbb{R}$. In this case, if $m = 3$, then topologically $\Sigma$ is either a sphere or a torus.
In the torus case, $M^3$ is flat.}
\end{theorem}
\begin{proof}
The first conclusion follows if we take $\lambda = 1$ in corollary $1$.  Let $N$ be the unit normal
vector field on $\Sigma$. For $x$ close to $\Sigma$ in $M$, consider the oriended distance function
 $d(x)=Sign(x)dist(x, \Sigma)$, where $Sign(x)$ is $1$ if $x$ is on one side of $\Sigma$; $Sign(x) = -1$ if $x$
 is on the other side of $\Sigma$. Then $d(x)$ is smooth near $\Sigma$ and let $\Sigma_t$ be the level set
of $d(x)$. Then for $t$ small, $\Sigma_t$ is a smooth family of hypersurfaces on $M$ and we have
$$\frac{d(H-f_n)}{dt} = -Ric(n, n)-|A|^2-f_{nn} = -Ric_f(n, n)-|A|^2 \leq 0.$$  Note that $\Sigma_0$
is totally geodesic and $f_n = H = 0$ at $t=0$. Therefore $$H-f_n \leq 0$$ for
all $t$ and $$\frac{d\int_{\Sigma_t}e^{-f}}{dt} = \int_{\Sigma_t}(H-f_n)e^{-f} \leq 0.$$ Since $\Sigma_0$ is
area-minimizing in its homology class, $\Sigma_t$ are all totally geodesic. By induction, one can easily
show that $M$ is isometric to the quotient of $\Sigma_0 \times \mathbb{R}$. Therefore
 $$f_n = H = 0, f_{nn} = \frac{\partial f_n}{\partial t} =0, Ric_{nn} = 0$$ for all $t$.

Now consider the case when $m = 3$. Let $e_1, e_2$ be a local orthonormal frame on $\Sigma_0$. Let $S$ be
the scalar curvature on $M$; $S_f = S +\Delta f$; $K_\Sigma$ be the Gaussian curvature on $\Sigma$. Since
$\Sigma_0$ is totally geodesic,
 $$2K_{\Sigma_0} = 2R_{1221}  = S - 2Ric_{nn} = S_f - f_{11}- f_{22} = S_f - \Delta_{\Sigma_0}f.$$ In
the above equality, we have used the fact that $f_{nn} = 0$.  Since $S_f\geq 0$,
the Gauss-Bonnet theorem says that $\Sigma_0$ is either a sphere or a torus. In the torus case, $S_f = 0$
everywhere, thus on $\Sigma$, $Ric + \nabla^2 f = 0$.  So $\Sigma$ is a $2$ dimensional steady soliton.
Thus the Gaussian curvature on $\Sigma$ is nonnegative. This means that $\Sigma$ and $M$ are flat.
\end{proof}

\section{\bf{An example}}
In \cite{[FS]}, R. Schoen and D. Fischer-Colbrie proved the following theorem:
\begin{theorem}[R. Schoen and D. Fischer-Colbrie]
Let $M$ be a complete oriented 3-manifold with nonnegative scalar curvature. Let $\Sigma$ be an oriented complete stable minimal surface in $M$, then
if $\Sigma$ is compact, then it is conformal to $\mathbb{S}^2$ or a torus $\mathbb{T}^2$; if $\Sigma$ is not compact, it is conformally covered by $\mathbb{C}$.
\end{theorem}

In view of Theorem 2, it is natural to ask whether we can weaken the condition in theorem 1 when $dim(M) = 3$. We will show that at least locally,
even if the Bakry-Emery Ricci curvature is nonnegative, the stability of a weighted stable minimal surface $\Sigma$ does not provide any information on the conformal structure on $\Sigma$.

Let $M^3$ be an oriented manifold with nonnegative Bakry-Emery Ricci curvature and $\Sigma$ be an oriented stable $f$-minimal surface in $M$.
In this section we will give an explicit example so that $\Sigma$ is hyperbolic.

\bigskip

Let $(\Sigma, ds_{\Sigma}^2)$ be a complete surface with curvature $-1$. Let $M = (-\frac{1}{2}, \frac{1}{2}) \times \Sigma$ and define metric on $M$ by $$ds^2 = dt^2 + g(t)ds_{\Sigma}^2.$$ Note that the metric on $M$ is not complete. Let $p\in M$ and consider a product chart $U\ni p$ such that $e_1=\frac{\partial}{\partial x_1}, e_2=\frac{\partial}{\partial x_2}$ are tangential to $\Sigma_t$ and $\frac{\partial}{\partial t} = e_3$ on $U$.
We may assume that $e_1, e_2, e_3$ is an orthogonal frame in $U$ and $ds_\Sigma^2(e_1, e_1) = ds_\Sigma^2(e_2, e_2)= 1$. Then $$\langle\nabla_{e_1}e_3, e_1\rangle=\langle\nabla_{e_2}e_3, e_2\rangle=\frac{1}{2}g'(t),$$  $$\langle\nabla_{e_1}e_3, e_2\rangle = \langle\nabla_{e_2}e_3, e_1\rangle = 0.$$ Therefore, $\nabla^{\Sigma_t}A = 0$ for all $t$.
 By Gauss equation, $$K_{\Sigma_t} - \frac{R_{1221}}{g^2} = \frac{A_{11}A_{22}}{g^2} = \frac{\langle\nabla_{e_1}e_3, e_1\rangle\langle\nabla_{e_2}e_3, e_2\rangle}{g^2}.$$ Since the Gaussian curvature $K_{\Sigma_t} = -\frac{1}{g}$,  $$R_{1221} = -g-\frac{1}{4}g'^2.$$
It is easy to see that $\nabla_{e_3}e_3 \equiv 0$, thus
\begin{equation}
\begin{aligned}
R_{1331} &= \langle\nabla_{e_1}\nabla_{e_3}e_3, e_1\rangle-\langle\nabla_{e_3}\nabla_{e_1}e_3,e_1\rangle\\&=-\langle\nabla_{e_3}\nabla_{e_1}e_3,e_1\rangle\\&=
-(e_3(\langle\nabla_{e_1}e_3, e_1\rangle)-|\nabla_{e_3}e_1|^2)\\&=
-\frac{1}{2}g''+\frac{1}{4}\frac{g'^2}{g}.
 \end{aligned}
\end{equation}
 From the same computation, we see that $R_{1332} = 0$.
By Codazzi equation, $$R_{1223} = (\nabla^{\Sigma_t}_{e_1}A)(e_2, e_2) - (\nabla^{\Sigma_t}_{e_2}A)(e_1, e_2) = 0.$$ Therefore

$$Ric_{11} = \frac{R_{1221}}{g} + R_{1331} = -1-\frac{1}{2}g'' = Ric_{22},$$ $$Ric_{33} = -\frac{g''}{g}+\frac{1}{2}(\frac{g'}{g})^2,$$
$$Ric_{12} = Ric_{13} = Ric_{23} = 0.$$

Let $f =f(t)$ be a function of $M$, then $$f_{11} = - \langle\nabla f, \nabla_{e_1}e_1\rangle=\frac{g'f'}{2}=f_{22},$$ $$f_{12}=f_{13}=f_{23}=0, f_{33} = f''.$$

Therefore $$Ric_f(e_1, e_1) = -1-\frac{g''}{2}+\frac{f'g'}{2}, Ric_f(e_3, e_3) = \frac{-2g''g+g'^2+2g^2f''}{2g^2}$$

 If $f=-2t^2, g = 1-2t^2$, then
one gets that $$Ric_f(e_2, e_2) = Ric_f(e_1, e_1) = 1+8t^2 \geq 0, Ric_f(e_3, e_3) = 4(\frac{1}{(1-2t^2)^2}-1)\geq 0.$$
Therefore, $M$ has nonnegative Bakry-Emery Ricci curvature.
Moreover, the second fundamental form and $Ric_f(e_3, e_3)$ vanish at $t = 0$. According to corollary 1, $\Sigma_0$ is a stable $f$-minimal surface in $M$. However, $\Sigma$ is hyperbolic.

\section{\bf{Applications to the noncompact case}}

Now consider the case when $\Sigma$ is noncompact. The following proposition follows from a simple cut-off argument:
\begin{prop}
{Let $(M^m, g, e^{-f}dv)$ be an oriented complete Riemannian manifold and $\Sigma$ be a complete noncompact oriented stable $f$-minimal hypersurface in $M$.
 If $Ric_f \geq 0$ on $\Sigma$ and that the weighted volume growth of $\Sigma$ with respect to its intrinsic distance to a point $p \in \Sigma$ satisfy $$V_{\Sigma, f}(B_p(r)) \leq Cr^2$$ for all large $r$ , then $\Sigma$ is totally geodesic and $Ric_f(n, n) = 0$.}
\end{prop}
\begin{proof}
Let $r$ be a distance function to $p\in M$. Given any $a > 1$, consider the cut-off function
\begin{equation}
\lambda(r) = \left\{
\begin{array}{rl} 1 & 0 \leq r \leq a \\
\frac{2\log a - \log r}{\log a} & a< r < a^2 \\ 0 & r\geq a^2.
\end{array}\right.
\end{equation}

 Define $V(r) = \int_{B_{\Sigma}(p,r)}e^{-f}$. Plugging this in the stability inequality in corollary $1$, we find that
\begin{equation}
\begin{aligned}
&\int_{B_{\Sigma}(a)}(Ric_f(n,n)+|A|^2)|\lambda|^2e^{-f} \\&\leq \int_{B_{\Sigma}(a^2)}|\nabla\lambda|^2e^{-f}\\&=\int_{a}^{a^2}\frac{V'(r)}{r^2\log^2a}dr
\\&=\frac{V(r)}{r^2\log^2a}|_{r=a}^{r=a^2}-\int_{a}^{a^2}V(r)(\frac{1}{r^2\log^2a})'dr\\&
\leq \frac{C}{\log^2a}+C\frac{1}{\log^2a}\int_{a}^{a^2}\frac{dr}{r} \\&
=O(\frac{1}{\log a}).
 \end{aligned}
\end{equation}
The proposition follows by taking $a\to \infty$.
\end{proof}

Now recall the following theorem in \cite{[WW]}\cite{[MW]}
\begin{lemma}
{Let $(M^m, e^{-f}dv)$ be a smooth metric measure space with $Ric_f \geq 0$, then along any miminizing geodesic starting from $x\in B_p(R)$ we have
$$\frac{J_f(x, r_2, \xi)}{J_f(x, r_1, \xi)}\leq e^{4A(R)}(\frac{r_2}{r_1})^{m-1}$$
for $0 < r_1<r_2<R$. In particular, for $0 < r_1<r_2<R$, the weighted area of the geodesic spheres satisfy $$\frac{ A_f(\partial B_x(r_2))}{A_f(\partial B_x(r_1))} \leq e^{4A(R)}\frac{r_2^{m-1}}{r_1^{m-1}}.$$ Here $A(R) = Sup_{x \in B_x(3R)}|f|(x)$} and $J_f(x, r, \xi)=e^{-f}J(x, r, \xi)$ is the $e^{-f}$ weighted volume in geodesic polar coordinates.
\end{lemma}

If $f$ is bounded, $Vol_f(B_x(r))$ has polynomial growth of order at most $m$.
\begin{prop}
{Let $(M^3, e^{-f}dv)$ be a smooth metric measure space with  $Ric_f \geq 0$ and $f$ is bounded.
 If $\Sigma$ is a complete weighted area-minizing hypersurface which is the boundary of least weighted area in $M$, then $\Sigma$ is totally geodesic and $Ric_f(n, n) = 0$.}
\end{prop}
\begin{proof}
According to lemma 1, the weighted volume of the geodesic sphere has at most quadratic growth.  Since $\Sigma$ is weighted area minimizing and is a
boundary of least weighted area in $M$, $vol_f(\Sigma\vdash B_x(r)) \leq A_f(\partial B_x(r))\leq Cr^2$. Proposition 3 follows from Proposition 2.
\end{proof}

\section{\bf{Application to complete 3-manifolds with nonnegative Bakry-Emery Ricci curvature}}

The classification of complete 3-manifolds with nonnegative Ricci curvature has been complete by various authors' works.
By using the Ricci flow, Hamilton \cite{[H1]}\cite{[H2]} classified all compact 3-manifolds with nonnegative Ricci curvature.
He proved that the universal cover is either diffeomorphic to $\mathbb{S}^3$, $\mathbb{S}^2\times\mathbb{R}$ or $\mathbb{R}^3$.
In the latter two cases, the manifold splits.

In \cite{[SY1]}, Schoen-Yau proved that a complete noncompact 3-manifold with positive Ricci curvature is diffeomorphic to the Euclidean space.
Anderson-Rodriguez \cite{[AR]} and Shi \cite{[Sh]} classified complete noncompact 3-manifolds with nonnegative Ricci curvature by assuming an upper bound of the
sectional curvature. Very recently, the author \cite{[L]} classified all complete noncompact 3-manifolds with nonnegative Ricci curvature.

In view of the results above, it is natural to ask what happens to a 3-manifold when the Bakry-Emery Ricci curvature is nonnegative.
Below is a partial classification when $f$ is bounded.

\begin{theorem}
Let $(M^3,g, e^{-f}dv)$ be a complete 3-manifold with bounded $f$
and $Ric_f\geq 0$.
\begin{itemize}
\item If $M$ is noncompact, then either $M$ is contractible or through each point in $M$, there exists a
totally geodesic surface with $Ric_f(n, n) = 0$.
If in addition the rank of $Ric_f$ is at least $2$ everywhere,
then the universal of $M$ splits as a Riemann product as $\Sigma \times \mathbb{R}$.
In particular, if the Bakry-Emery Ricci curvature is positive, then $M$ is contractible.
\item If $M$ is compact, then either it is a quotient of $\mathbb{S}^3$ or the universal cover splits as a product $\Sigma\times \mathbb{R}$.
\end{itemize}
In each splitting case, $\Sigma$ is conformal to $\mathbb{S}^2$ or $\mathbb{C}$ and $f$ is constant along the $\mathbb{R}$ factor.\end{theorem}
\begin{proof}
First we consider the case when $M$ is noncompact.
The argument is similar to \cite{[SY1]}\cite{[L]}.
Assume $M$ is simply connected, if $\pi_2(M) \neq 0$,  according to Lemma 2 in \cite{[SY1]}, $M$ must have at least
 two ends. From Lichnerowicz's extension of the Cheeger-Gromoll splitting theorem \cite{[Lc]}, the universal cover
 splits. So we assume $\pi_2(M) = 0$. Therefore, the universal cover of $M$ is contractible. If $M$ is not simply
connected, Schoen and Yau \cite{[SY1]} proved that $\pi_1(M)$ must have no torsion elements. Thus, after
 replacing $M$ by a suitable covering, we may assume that $\pi_1(M) = \mathbb{Z}$ and that $M$ is orientable.

 Recall lemma $2.2$ in \cite{[An1]} by Anderson:
 \begin{lemma}(Anderson)
{ Let $M$ be a complete Riemannian manifold with finitely generated homology $H_1(M ,\mathbb{Z})$. Then any
non-zero line $\mathbb{R}\cdot\alpha, \alpha\in H_1(M, \mathbb{Z})$ gives rise to a complete homologically area-minimizing
hypersurface $\Sigma_{\alpha}$, which is the boundary of least area in a cover $\mathbb{Z} \rightarrow \overline{M}\rightarrow M$. Moreover,
the volume growth of $\Sigma_{\alpha}$ satisfies $vol(\Sigma\vdash B^{\overline M}(r)) \leq vol(\partial B^{\overline M}(r))$ and the intersection
number $I(\Sigma, \alpha) \neq 0$.
}
\end{lemma}
The proof of the above lemma in \cite{[An1]} can be carried out without any modification to weighted volume case.
 Taking $\alpha$ to be the generator of $H_1(M, \mathbb{Z})$, we can find a complete oriented boundary $\Sigma$ of least weighted area in the universal cover
$\tilde{M}$.  By proposition 3, $\Sigma$ is totally geodesic and $Ric_f(n, n) = 0$.  If $Ric_f > 0$ on $M$, then this is a contradiction.

Now consider the case when $Ric_f \geq 0$.
We shall use a perturbation argument in \cite{[E]}\cite{[L]}.
For any point $p\in M$, consider a family of metric $g(t) = e^{2t\lambda}g_0$, where $\lambda=\lambda(x)$ is a function on $M$.
Let $(U, g_{ij},x_i)$ be a normal coordinate for $g_0$ at $p$ such that $\frac{\partial}{\partial x_i} = e_i$. We have $$f^t_{ij} = e_je_i(f) -(\nabla^t_{e_j}e_i) f,$$ $$\Gamma^s_{ij}(g(t)) = \frac{1}{2}g^{sl}(t)(\frac{\partial g_{il}(t)}{\partial x_j} + \frac{g_{jl}(t)}{\partial x_i} - \frac{\partial g_{ij}(t)}{\partial x_l}).$$ Then at $p$,$$\Gamma^s_{ij}(g(t)) = t(\lambda_j\delta_{is}+\lambda_i\delta_{js}-\lambda_s\delta_{ij}).$$
Therefore,
\begin{equation}
\begin{aligned}
f^t_{ij}-f_{ij} &= -\Gamma^s_{ij}(g(t))f_s \\&= t(f_s\lambda_s\delta_{ij}-\lambda_if_j-\lambda_jf_i)\\& \geq-3t|\nabla f||\nabla\lambda|.
\end{aligned}
\end{equation}
Let $m = dim(M) = 3$.
Recall that
$$Ric^t(v, v) = (Ric(v, v) - t(m-2)\lambda_{vv} - t\Delta\lambda + t^2(m-2)(v(\lambda)^2-|\nabla\lambda|^2))$$ for $|\nu|_{g_0} = 1.$
Let $r(x) = dist(x, p)$ on $M$.
For a very small $R > 0$, consider the function
$\rho = R-r$ for $\frac{R}{2}< r < R$. Then we extend $\rho$ to be a positive smooth function for $0 \leq r < \frac{R}{2}$.
Define $\lambda = -\rho^5$.

Now $$\nabla^2 (\rho^5)(v, v) = 20\rho^3v(\rho)^2+ 5\rho^4\nabla^2(\rho)(v, v).$$ For $aR < r < R$, we have
\begin{equation}
\begin{aligned}
Ric^t(v, v)+f^t_{vv} &\geq Ric^0(v, v)+f^0_{vv} + 20t\rho^3+ 5t\rho^4(\Delta \rho+ \\&(m-2)\nabla^2 (\rho)(v, v)) -25(m-2)t^2\rho^8-15t\rho^4|\nabla f|.
\end{aligned}
\end{equation} Using the fact that the manifold is almost Euclidean near $p$, for small $R$, we have $$|\Delta \rho + (m-2)\nabla^2 \rho(v, v)|  \leq \frac{9(2m-3)}{8(R-\rho)}.$$ Therefore, there exists small $R > 0$ such that for all small $t$, $Ric^t_f(v, v) > 0$ in an annulus $B_p(R)\backslash B_p(aR)$ for $a = \frac{7}{8}$. The metric remains the same outside $B_p(R)$. The deformation is $C^4$ continuous with respect to the metric and $C^{\infty}$ with respect to $t$.

Let $\gamma$ be a closed curve in $M$ which represents the generator of $\pi_1(M)$. We can apply the perturbation finitely many times such that $Ric_f > 0$ on $\gamma$ and $Ric_f$ is nonnegative on $M$ except a small neighborhood $U$ of $p$. Then for the perturbed metric $g_t$, we can apply lemma 2 to obtain a complete oriented boundary $\Sigma$ of least weighted area in the universal cover
$\tilde{M}$.  Since $g_t$ is uniformly equivalent to $g_0$, we can show $\Sigma_t$ has quadratic weighted volume growth. Let $q \in \Sigma_t$, then for any $r > 0$,
\begin{equation}
\begin{aligned}
vol_{g(t)}(\Sigma_t\vdash B_{g(t)}(q, \tilde{M})(r)) &\leq vol_{g(t)}(\Sigma_t\vdash B_{g(0)}(q, \tilde{M})(Cr)) \\&\leq vol_{g(t)}(\partial B_{g(0)}(q, \tilde{M})(Cr))\\&\leq Cvol_{g(0)}(\partial B_{g(0)}(q, \tilde{M})(Cr))\\&\leq C_1r^2.
\end{aligned}
\end{equation}

If $\Sigma_t$ does not intersect the preimage of $U$ in $\tilde{M}$, then on $\Sigma_t$, $Ric_f \geq 0$ and $Ric_f>0$ at $\Sigma_t \cap \gamma$. This contradicts proposition 2.

For each $\Sigma_t$, we can find deck transformation $l_t$ on $\tilde{M}$ such that $l_t(\Sigma_t)$ intersects the preimage of $U$ at some fixed compact set in $\tilde{M}$.
Therefore, if we shrink the size of the neighborhood of $p$ and let $t\to 0$ sufficiently fast, a subsequence of $\Sigma_t$ will converge to a weighted area minimizing surface $\Sigma$ satisfying $$vol_{g(0)}(\Sigma\vdash B_{g(0)}(q, \tilde{M})(r)) \leq Cr^2.$$
Thus, by proposition 2,  $\Sigma$ is totally geodesic and $Ric_f(n, n) = 0$.  Since $p$ is arbitrary,
though each point there exists a totally geodesic surface with $Ric_f(n, n) = 0$.

\bigskip

Now we use the assumption that the rank of $Ric_f$ is at least $2$ everywhere.
Then through each point $p\in \tilde{M}$, there exists a unique totally geodesic surface.
Therefore we have a foliation on $\tilde{M}$. We can parametrize the surfaces as $\Sigma_t$.

Let $N$ be the unit normal vector and $\lambda N$ be the variational vector field of $\Sigma_t$.
Since the smooth family of surfaces $\Sigma_t$ never intersect with each other, $\lambda$ is nonnegative. A simple computation shows that the
variational vector field of these totally geodesic surfaces satisfies $$\Delta\lambda + \lambda Ric(n, n) = 0.$$ Since
 $$H = f_n = 0,$$
\begin{equation}
 \begin{aligned}
0&=\frac{df_n}{dt}\\&= \lambda f_{nn}+\langle \nabla f, \nabla_{\lambda N} N\rangle\\&
=\lambda f_{nn} + \sum\limits_{i=1}^{2}\langle\nabla f, e_i\rangle\langle e_i, \nabla_{\lambda N} N\rangle\\&
=\lambda f_{nn} - \langle \nabla f, \nabla\lambda\rangle.
 \end{aligned}
\end{equation}
In the above computation, $e_i$ is an orthonormal frame on an open set of $\Sigma$.

 But $$0 = Ric_f(n, n) = Ric(n, n) + f_{nn},$$
 thus we have $$\Delta_f\lambda = \Delta\lambda - \langle\nabla\lambda, \nabla f\rangle= 0$$ on $\Sigma$.

The lemma below is close to corollary 1 in \cite{[CY]}.
\begin{lemma}
For a smooth metric measured space $(M, g, e^{-f}dv)$ with quadratic weighted volume growth, if $\lambda$ is a positive function which satisfies $\Delta_f\lambda = 0$, then $\lambda$ is a constant.
\end{lemma}
\begin{proof}
 Let $\lambda = e^h$, then $$\Delta h + |\nabla h|^2 -\langle\nabla h, \nabla f\rangle=0.$$
Let $\varphi$ be a cut-off function, we find
$$\int\varphi^2\Delta he^{-f} + \int\varphi^2|\nabla h|^2e^{-f} - \int\varphi^2\langle\nabla h, \nabla f\rangle e^{-f} = 0.$$
By integration by parts, $$\int\varphi^2(\Delta h)e^{-f} = -\int h_i2\varphi\varphi_ie^{-f}+\int h_i\varphi^2f_ie^{-f}.$$
Therefore $$\int\varphi^2|\nabla h|^2e^{-f} = 2\int\varphi_ih_i\varphi e^{-f}\leq 2(\int\varphi^2|\nabla h|^2e^{-f})^{\frac{1}{2}}(\int|\nabla\varphi|^2e^{-f})^{\frac{1}{2}}.$$ Thus $$\int\varphi^2|\nabla h|^2e^{-f} \leq 4\int|\nabla\varphi|^2e^{-f}.$$
Now we can use the same cut-off function in proposition 2 to show that $\nabla h \equiv 0$. Thus $\lambda$ is a constant.
\end{proof}

Since $\lambda$ is nonnegative, by lemma 3, $\lambda$ is constant. After a reparametrization of $\Sigma_t$, we may assume $\lambda = 1$.
Now for $X\in T\Sigma_t$, $\nabla_XN = 0$, since $\Sigma_t$ is totally geodesic. Since $\lambda$ is a constant, we may assume $[X, N] = 0$. $\langle\nabla_NN, X\rangle = -\langle N, \nabla_NX\rangle
=-\langle N, \nabla_XN\rangle = 0$. Thus $\nabla N \equiv 0$.
Therefore $M$ is locally isometric to $\Sigma\times\mathbb{R}$.
$f$ is constant along the $\mathbb{R}$ factor, since $f_n = 0$.

\bigskip

Now consider the case when $M$ is compact.  If the universal cover is compact, then according to Perelman's
solution to the Poincare conjecture, $M$ is covered by $\mathbb{S}^3$. If the universal cover $\tilde{M}$ is noncompact, then according
to Theorem 6.6 in \cite{[WW]},  $\tilde{M}$ splits as a product $\Sigma\times \mathbb{R}$.

Finally, we show that in the splitting case, $\Sigma$ is conformal to $\mathbb{C}$ or $\mathbb{S}^2$.
There are two methods to do this. Note that on $\Sigma$, $$Ric_\Sigma + \nabla^2 f \geq 0.$$
Consider the conformal change of the metric $\tilde{g} = e^{-f}g$ on $\Sigma$, then the tensor $$Ric_\Sigma(\tilde{g}) = Ric_\Sigma(g)
 + \frac{1}{2}(\Delta_\Sigma f) g\geq 0.$$
As $f$ is bounded, $\tilde{g}$ is complete. Since $\Sigma$ is simply connected, $\Sigma$ is conformal to $\mathbb{C}$ or $\mathbb{S}^2$.

The second way is this: By lemma 1, the weighted volume growth of $\Sigma$ is at most quadratic. Since $f$ is bounded, the volume growth of $\Sigma$
is at most quadratic. If $\Sigma$ is conformal to the Poincare disk, then there exists a nontrivial bounded harmonic function  on $\Sigma$.
But according to corollary 1 in \cite{[CY]}, the function is a constant. This is a contradiction.

\end{proof}
\begin{remark}
The bounded condition of $f$ cannot be dropped in the above theorem.  For example, consider the warped product metric
$ds^2 = dt^2 + g(t)ds_{\Sigma}^2$ on
$M = \mathbb{S}^2 \times \mathbb{R}$. Here $ds_{\Sigma}^2$ is the standard metric on $\mathbb{S}^2$ with curvature $1$.
Consider an orthogonal frame $e_1, e_2, e_3$ on $M$ such that $ds_{\Sigma}^2(e_1, e_1) = ds_{\Sigma}^2(e_2, e_2) = 1$ and $\frac{\partial}
{\partial t} = e_3$.

If we take $f$ as a function of $t$ on $M$, then by similar computations in section $4$, we see
$$Ric_f(e_1, e_1) = 1-\frac{g''}{2}+\frac{f'g'}{2}, Ric_f(e_3, e_3) = \frac{-2g''g+g'^2+2g^2f''}{2g^2}.$$

If $f(t) = t^2$, $g(t) = e^t$, then one can check that $Ric_f > 0$, however, $M$ is not a Riemann product or a contractible manifold.

\end{remark}

\end{document}